\documentclass[a4paper]{amsart}
\pdfoutput=1
\usepackage[hmargin=30mm, vmargin=25mm, includefoot,twoside]{geometry}
\usepackage[shortcuts]{extdash}
\usepackage{txfonts}
\usepackage{multicol}
\usepackage{xspace}
\usepackage{verbatim}
\usepackage{amssymb, mathtools, amsthm}
\usepackage{enumitem}
\usepackage{pinlabel}
\usepackage{graphicx}
\usepackage{hyperref}

\allowdisplaybreaks

\title{Cogrowth for group actions with strongly contracting elements}
\author{Goulnara N. Arzhantseva}
\author{Christopher H. Cashen}
\address{Universit\"at Wien, Fakult\"at f\"ur Mathematik\\
Oskar-Morgenstern-Platz 1, 1090 Wien, \"Osterreich.}
\email{\href{mailto:goulnara.arzhantseva@univie.ac.at}{goulnara.arzhantseva@univie.ac.at}, \href{mailto:christopher.cashen@univie.ac.at}{christopher.cashen@univie.ac.at}}

\thanks{This research was partially supported by the European Research
  Council (ERC) grant of Goulnara Arzhantseva, ``ANALYTIC'' grant
  agreement no.\ 259527. The second author is supported by the Austrian Science Fund (FWF): P~30487\=/N35.} 

\hypersetup{
    pdftitle={Cogrowth for group actions with strongly contracting elements},    
    pdfauthor={Goulnara N. Arzhantseva and Christopher H. Cashen},     
    pdfkeywords={Cogrowth, exponential growth, divergence type, contracting element, mapping class groups, right-angled Artin groups, snowflake groups, CAT(0) groups}, 
    colorlinks=true,       
    linkcolor=black,          
    citecolor=black,        
    filecolor=black,      
    urlcolor=black           
}

\theoremstyle{plain}
\newtheorem{theorem}{Theorem}[section]
\newtheorem{lemma}{Lemma}[section]
\newtheorem{proposition}{Proposition}[section]
\newtheorem{corollary}{Corollary}[section]
\newtheorem{question}{Question}

\def\makeautorefname#1#2{\expandafter\def\csname#1autorefname\endcsname{#2}}
\let\fullref\autoref

\makeautorefname{theorem}{Theorem} 
\makeautorefname{lemma}{Lemma} 
\makeautorefname{proposition}{Proposition} 
\makeautorefname{corollary}{Corollary} 
\makeautorefname{section}{Section}
\makeautorefname{subsection}{Section}
\makeautorefname{subsubsection}{Section}
\makeautorefname{question}{Question}

\makeatletter 
\let\c@lemma=\c@theorem 
\makeatother
\makeatletter 
\let\c@proposition=\c@theorem 
\makeatother
\makeatletter 
\let\c@corollary=\c@theorem 
\makeatother
\makeatletter 
\let\c@question=\c@theorem 
\makeatother
\mathtoolsset{centercolon} 

\newcommand{\bdry}{\partial} 
\newcommand{\act}{\curvearrowright} 
\newcommand{\from}{\colon\thinspace} 
\newcommand{\bp}{o}

  \newcommand{\emul}{\stackrel{{}_\ast}{\asymp}}
  \newcommand{\gmul}{\stackrel{{}_\ast}{\succ}}
  \newcommand{\lmul}{\stackrel{{}_\ast}{\prec}}
  \newcommand{\eadd}{\stackrel{{}_+}{\asymp}}
  \newcommand{\gadd}{\stackrel{{}_+}{\succ}}
  \newcommand{\ladd}{\stackrel{{}_+}{\prec}}
  \newcommand{\laddmul}{\prec}
  \newcommand{\eaddmul}{\asymp}
  \newcommand{\gaddmul}{\succ}

\DeclareMathOperator{\diam}{diam}

\newcommand{\mapsfrom}{\mathrel{\reflectbox{\ensuremath{\mapsto}}}}

\newcommand{\E}{\mathcal{E}}
\newcommand{\W}{\mathcal{W}}
\newcommand{\X}{\mathcal{X}}
\newcommand{\Y}{\mathcal{Y}}
\newcommand{\Z}{\mathcal{Z}}
\newcommand{\YY}{\mathbf{{Y}}}

\newcommand{\spthree}{\hyperref[SP3]{(SP~3)}\xspace}
\newcommand{\spfour}{\hyperref[SP4]{(SP~4)}\xspace}
\newcommand{\pzero}{\hyperref[P0]{(P~0)}\xspace}
\newcommand{\pone}{\hyperref[P1]{(P~1)}\xspace}
\newcommand{\order}{\sqsubset}

\begin{document}
\begin{abstract}
Let $G$ be a group acting properly by isometries
and with a strongly contracting element on a
geodesic metric space.
  Let $N$ be an infinite normal subgroup of $G$, and let $\delta_N$ and $\delta_G$
  be the growth rates of $N$ and $G$ with respect to the pseudo-metric
  induced by the action.
We prove that if $G$ has purely exponential growth with respect to the
pseudo-metric then 
$\delta_N/\delta_G>1/2$.
Our result applies to suitable actions of hyperbolic groups,  right-angled Artin groups and other CAT(0) groups,
  mapping class groups, snowflake groups, small cancellation groups,
  etc.
  This extends Grigorchuk's original result on free groups with
  respect to a word metrics and
  a recent result of Jaerisch, Matsuzaki, and Yabuki on groups acting on hyperbolic spaces to
  a much wider class of groups acting on spaces that are not necessarily hyperbolic.

\end{abstract}

\maketitle


\section{Introduction}

We consider the exponential \emph{growth rate} $\delta_G$ of the orbit of a group $G$ acting properly on a geodesic metric space $X$.
In various notable contexts 
this asymptotic invariant is
related to  the Hausdorff dimension of the
limit set of $G$ in $\bdry X$ and to analytical and dynamical properties
of $G\backslash X$ such as the spectrum of the
Laplacian, divergence rates of random walks, volume entropy, and ergodicity of the
geodesic flow.

In some cases of special interest, the value of half the growth rate
of the ambient space $X$ is distinguished.
For example, when  $X=\mathbb{H}^n$ and $H$ is a torsion free discrete
group of isometries of $X$, the 
Elstrodt-Patterson-Sullivan formula \cite{MR882827} for the bottom of the spectrum of the
Laplacian of $H\backslash X$ has a phase change when the ratio of
$\delta_H$ to the volume entropy of $X$ is $1/2$.
Similarly, if $X$ is a Cayley tree of a finite rank free group $F_n$
and $H$ is a subgroup, then the Grigorchuk cogrowth formula
\cite{MR599539} for the spectral radius of 
$H\backslash X$ has a phase change at $\delta_H/\delta_{F_n}=1/2$.
Our main result says that, in great generality, normal subgroups land decisively on one side of
this distinguished value:

\begin{theorem}\label{maintheorem}
  Suppose $G$ is a group acting properly by
  isometries on a 
  geodesic metric space $X$ with a strongly contracting element and
  with purely exponential growth.
  If $N$ is an infinite normal subgroup of $G$ then
  $\delta_N/\delta_G>1/2$, where the growth rates $\delta_G$ and
  $\delta_N$ are computed with respect to
  $G\act X$. 
\end{theorem}

The ratio $\delta_N/\delta_{G}$ is known as the \emph{cogrowth} of
$Q:=G/N$.
The hypotheses will be explained in detail in the next section.
Briefly, the existence of a strongly contracting element means that
some element of $G$ acts hyperbolically on $X$, though $X$ itself need
not be hyperbolic, and pure exponential
growth is guaranteed if the action has a strongly contracting element
and an orbit of $G$ in $X$ is not too badly
distorted.

In negative curvature, the strict lower bound on cogrowth has been
shown in various special cases \cite{MR1792293,MR2166367,MR2891734,Jae15}.
For $X=G=F_n$, the strict lower bound on cogrowth 
is due to Grigorchuk
\cite{MR599539}.

Grigorchuk and de la Harpe \cite[page 69]{GriDeL97} (see also
\cite[Problem 36]{MR1786869}) asked whether the strict
lower cogrowth bound also holds when $F_n$ is replaced by a non-elementary Gromov hyperbolic
group, and $X$ is one of its Cayley graphs.
This long-open problem was recently answered affirmatively by Jaerisch,
Matsuzaki, and Yabuki \cite{MatYabJae15} (see also a survey by
Matsuzaki \cite{Mat17}).
Their result applies more generally to groups of divergence type
acting on hyperbolic spaces. 
\fullref{maintheorem} gives an alternative proof of the positive
answer to Grigorchuk and de la Harpe's question, and goes much beyond.
In comparison, Jaerisch,
Matsuzaki, and Yabuki's result applies to more general actions if one restricts to actions on  \emph{hyperbolic
  spaces}, while \fullref{maintheorem} applies to many renowned 
non-hyperbolic examples. 

\begin{corollary}
  For the following $G\act X$, for every infinite normal subgroup $N$ of
  $G$ we have $\delta_N/\delta_G>1/2$.
  \begin{enumerate}
    \item $G$ is a non-elementary hyperbolic group acting cocompactly
      on a hyperbolic space $X$.\label{item:hyperbolic}
      \item $G$ is a relatively hyperbolic group, and $X$ is hyperbolic
        such that $G\act X$ is cusp uniform and satisfies the
        parabolic gap condition.\label{item:relhyperbolic}
  \item $G$ is a right-angled Artin group defined by a finite simple graph that
  is neither a single vertex nor a join, and $X$ is the universal
  cover of its Salvetti complex.\label{item:raag}
  \item $X$ is a CAT(0) space, and $G$ acts cocompactly with
    a rank 1 isometry on $X$.\label{item:cat0}
  \item $G$ is the mapping class group of a surface of genus $g$ and
    $p$ punctures, with $6g-6+2p\geqslant 2$, and $X$
    is the Teichm\"uller space of the surface with the Teichm\"uller
    metric.\label{item:mcg}
  \end{enumerate}
\end{corollary}

Results \eqref{item:raag}-\eqref{item:mcg} are new, only known as consequences of
\fullref{maintheorem}.
Further new examples include wide classes of snowflake groups~\cite{ArzCasTao15}
and of infinitely presented graphical and classical small
cancellation groups~\cite{ArzCasGrua}, hence, many so-called infinite `monster' groups.

The generality of \fullref{maintheorem} is striking.
Previous successes in showing the strict lower bound on cogrowth have relied on
fairly sophisticated results concerning Patterson-Sullivan measures on
the boundary of a hyperbolic space or ergodicity of the geodesic flow
on $G\backslash X$.
These tools are not available in our general setting.
Instead, we use the geometry of the group action directly to estimate orbit growth. 
The idea of our argument is as follows.
\begin{enumerate}
\item If $G$ contains a strongly contracting element for $G\act X$
  then so does every infinite normal subgroup $N$ of $G$. Let $c\in N$
  be such an element.
  \item By passing to a high power of $c$, if necessary, we may assume
    that its translation length is much larger than the constants
    describing its strong contraction properties.
    In this case the growth $\delta_{[c]}$ of the set $[c]$ of
    conjugates of $c$ is exactly $\delta_G/2$.
    \item A `tree's worth'
      of copies of $[c]$ injects into the normal closure
      $\langle\langle c\rangle\rangle$ of $c$, which is a subgroup of
      $N$. It follows that the growth rate of $\langle\langle
      c\rangle\rangle$, hence of $N$, is strictly
      greater than $\delta_{[c]}=\delta_G/2$.
      In this step we use the
      `hyperbolicity' of the action of $c$, as quantified by strong
      contraction, to provide geometric separation
      between copies of $[c]$.
\end{enumerate}

We used this strategy in  our paper with Tao \cite{ArzCasTao15} (see
also references therein) to prove
\emph{growth tightness} of $G\act X$ for actions having a strongly
contracting element. The key point was to estimate the growth rate of
the quotient of $G$ by the normal closure of $c$.
We chose a section $A$ of the quotient map and built a tree's worth of
copies of it by translating by a high power of $c$.
By construction, the set $A$ did not contain words containing high
powers of $c$ as subwords, so
translates of $A$ by powers of $c$ were geometrically separated.
There is a serious difficulty in applying step (3) for cogrowth,
because $[c]$ \emph{does} contain words with arbitrarily large powers
of $c$ as subwords. Indeed, any word of $G$ can occur as a subword of
an element of $[c]$, so we do not get the same nice geometric
separation as hoped for in step (3), and consequently our abstract tree's
worth of copies of $[c]$ does not inject into $G$.
We overcome this difficulty by quantifying how this mapping
fails to be an injection.
We show there is asymptotically at least half of $[c]$ for which the
map is an injection, and we use this half of $[c]$ to complete step (3).

For an example where the conclusion of the theorem does not hold,
consider the group $G=F_2\times F_2$ acting on its Cayley graph $X$ with
respect to the generating set $(S\cup
1)\times(S\cup 1)$, where $S$ is a free generating set of $F_2$. 
The $F_2$ factors are normal and have growth rate exactly half the
growth rate of $G$.
The action $G\act X$ does not have a strongly contracting
element.

\bigskip

We thank the referee for the careful reading and helpful comments.

\section{Preliminaries}

We write $x\lmul y$, $x\ladd y$, or $x\laddmul y$ if there is a universal constant $C>0$ such that
$x<Cy$, $x<y+C$, or $x<Cy+C$, respectively. We define $\gmul$, $\gadd$,
$\gaddmul$, $\emul$, $\eadd$, and $\eaddmul$ similarly.

Throughout, we let $(X,d,\bp)$ be a based geodesic metric space and let
$G$ be a group acting isometrically on $X$.
For $Y\subset X$ and $r\geqslant 0$, let $B_r(Y):=\{x\in X\mid \exists y\in
Y,\, d(x,y)<r\}$ and $\bar B_r(Y):=\{x\in X\mid \exists y\in
Y,\, d(x,y)\leqslant r\}$. Let $B_r:=B_r(\bp)$, and let $S_r^\Delta:=B_{r+\Delta}-B_{r}$.

There are induced pseudo-metric and semi-norm on $G$ given by
$d(g,h):=d(g.\bp,h.\bp)$ and  $|g|:=d(\bp,g.\bp)$.

\subsection{Growth}
The \emph{(exponential) growth rate} of a subset $Y\subset X$ is:
\[\delta_Y:=\limsup_{r\to\infty}\frac{\log \# Y\cap
    \bar B_r}{r}\]

The \emph{Poincar\'e series} of a countable subset $Y$ of $X$ is:
\[\Theta_Y(s):=\sum_{y\in Y}\exp(-sd(y,\bp))\]

For any $\Delta>0$ we also consider the series:
\[\Theta^{S,\Delta}_Y(s):=\sum_{i=0}^\infty (\#Y\cap S_{\Delta
    i}^{\Delta(i+1)})\exp(-s\Delta i)\]

\[\Theta^{B,\Delta}_Y(s):=\sum_{i=0}^\infty (\#Y\cap \bar B_{\Delta
    i})\exp(-s\Delta i)\]
The series $\Theta^{B,\Delta}_Y(s)$ and $\Theta^{S,\Delta}_Y(s)$ agree
with $\Theta_Y(s)$ up to multiplicative error depending on
$\Delta$ and $s$, so they all converge and diverge together.
Now, $\Theta_Y(s)$ converges for $s>\delta_Y$ and diverges for $s<\delta_Y$.
The set $Y$ is said to be \emph{divergent}, or \emph{of divergent
  type}, if $\Theta_Y(s)$ diverges at $s=\delta_Y$.

We say that $Y\subset X$ has \emph{purely exponential growth} if there
exist $\delta>0$ and $\Delta>0$ such that $\# Y\cap S_r^\Delta \emul
\exp(\delta r)$. Recall this means there is a constant $C>0$,
independent of $r$, such that $\exp(\delta r)/C\leqslant \# Y\cap S_r^\Delta
\leqslant C \exp(\delta r)$.

An action $G\act X$ is \emph{(metrically) proper} if for all $x\in X$ and
$r\geqslant 0$ the set $\{g\in G\mid d(x,g.\bp)\leqslant r\}$ is
finite.
When $G\act X$ is proper we extend all the preceding definitions to
subsets $H$ of $G$ by taking $Y=H.\bp$, e.g.:
\[\delta_H:=\limsup_{r\to\infty}\frac{\log \#
    H.\bp\cap\bar B_r}{r}=\limsup_{r\to\infty}\frac{\log\#\{h\in
    H\mid |h|\leqslant r\}}{r}\]

When $G\act X$ is cocompact, or, more generally, has a quasi-convex
orbit, the growth of $\#S_r^\Delta\cap G.\bp$ is
coarsely sub-multiplicative, which, when $\delta_G>0$, implies an exponential lower bound
on $\#S_r^\Delta\cap G.\bp$.
Conversely, if $G\act X$ contains a strongly contracting element then
the growth of $\#S_r^\Delta\cap G.\bp$ is coarsely
super-multiplicative, which implies the corresponding exponential
upper bound.
For instance, Coornaert \cite{MR1214072} proved that a
quasi-convex-cocompact, exponentially growing subgroup of a hyperbolic
group has purely exponential growth.
More generally, in \cite{ArzCasTao15} we introduced the following condition that
implies the pseudo-metric induced by a group action behaves like a
word metric for growth purposes: the \emph{complementary
  growth} of $G\act X$ is the growth rate of the set of points of $G.\bp$ that
can be reached from $\bp$ by a geodesic segment in $X$ that stays completely
outside of a neighborhood of $G.\bp$, except near its endpoints.
We say that $G\act X$ has \emph{complementary growth gap} if the
complementary growth is 
strictly less than $\delta_G$.
Yang \cite{Yan16} proved that if $G$ acts properly with a strongly contracting element and
$0<\delta_G<\infty$ then 
complementary growth gap implies purely exponential growth.

For relatively hyperbolic groups the complementary growth gap
specializes to the \emph{parabolic growth gap} of \cite{DalPeiPic11},
which requires that the growth of parabolic subgroups of a relatively
hyperbolic group is strictly less than the growth rate of the whole
group.
For another non-cocompact example, we showed in \cite{ArzCasTao15} that the action
of the mapping class group of a hyperbolic surface on its
Teichm\"uller space has complementary growth gap.

For a non-example, consider the integers $\mathbb{Z}$ acting parabolically on the
hyperbolic plane.
Hyperbolic geodesics connecting $\bp$ to $n.\bp$ for large $n$ travel
deeply into a horoball at the fixed point of $\mathbb{Z}$ on
$\bdry\mathbb{H}^2$, far from the orbit of $\mathbb{Z}$.
Although $\mathbb{Z}$ has $0$ exponential growth in any word metric,
in terms of this action on $\mathbb{H}^2$ it has exponential growth
due entirely to the distortion of the orbit.

\subsection{Contraction}
A subset $Y$ of $X$ is \emph{$C$--strongly contracting}, for a
`contraction constant' $C\geqslant 0$, if for all $x,\,x'\in X$, if $d(x,x')\leqslant
d(x,Y)$ then the diameter of $\pi_Y(x)\cup\pi_Y(x')$ is at most $C$,
where $\pi_Y(x):=\{y\in Y\mid d(x,y)=d(x,Y)\}$.
A set is called \emph{strongly contracting} if there exists a $C\geqslant
0$ such that it is $C$--strongly contracting.
The \emph{projection distance in $Y$} is $d^\pi_Y(x,x'):=\diam \pi_Y(x)\cup\pi_Y(x')$.
We extend these definitions to sets $Z\subset X$ by
$\pi_Y(Z):=\cup_{z\in Z}\pi_Y(z)$ and $d^\pi_Y(Z,Z'):=\diam \pi_Y(Z)\cup\pi_Y(Z')$.

Strong contraction of $Y$ is equivalent \cite[Lemma~2.4]{ArzCasTao15}
to the \emph{bounded geodesic image property}:
  For all $C\geqslant 0$ there exists $C'\geqslant C$ such that 
  if $Y$ is $C$--strongly contracting then
 for every geodesic $\gamma$ in $X$, if $\gamma\cap
 B_{C'}(Y)=\emptyset$ then $\diam\pi_Y(\gamma)\leqslant C'$.

  \begin{corollary}\label{BGI}
   Suppose $Y$ is $C$--strongly contracting and $C'$ is as above.
    Suppose $\gamma$ is a geodesic defined on an interval
 $[a,b]$, possibly infinite.
 Let $t_0:=\inf\{t\mid d(\gamma(t),Y)<C'\}$, and let $t_1:=\sup\{t\mid
 d(\gamma(t),Y)<C'\}$.
 Then $\diam\pi_Y(\gamma([a,t_0]))\leqslant C'$ and
 $\diam\pi_Y(\gamma([t_1,b]))\leqslant C'$, while $\gamma([t_0,t_1])\subset
 \bar B_{3C'}(Y)$. If $a$ and $b$ are finite and $\diam \pi_Y(\gamma(a))\cup\pi_Y(\gamma(b))>C'$ then
 $\pi_Y(\gamma(a))\subset \bar B_{2C'}(\gamma(t_0))$ and  $\pi_Y(\gamma(b))\subset \bar B_{2C'}(\gamma(t_1))$.
 \end{corollary}
 
An infinite order element $c\in G$ is said to be a \emph{strongly contracting element}
for $G\act X$
if the set  $\langle c\rangle.\bp$ is strongly contracting.
In this case $\mathbb{Z}\to X:i\mapsto c^i.\bp$ is a quasi-isometric embedding and $c$
is contained in a maximal virtually cyclic subgroup $E(c)$.
This subgroup, which is alternately known as the \emph{elementarizer}
or \emph{elementary closure} of $c$, can also be characterized as the
maximal subgroup consisting of elements $g\in G$ such that
$g^{-1}\langle c\rangle g$ is at bounded Hausdorff distance from
$\langle c\rangle$.
Since $E(c).\bp$ is coarsely equivalent to $\langle c\rangle.\bp$, the set
$E(c).\bp$ is also strongly contracting.
Note that $E(c)=E(c^n)$ for every
$n\neq 0$. Thus, when considering $E(c).\bp$, we can pass to powers of $c$ freely without changing the set
$E(c^n).\bp$, and in particular without changing its contraction
constant.

For a strongly contracting element $c$, let
$\E:=E(c).\bp$, and let $\YY$ be the collection of distinct
$G$--translates of $\E$. Bestvina, Bromberg, and Fujiwara \cite{BesBroFuj15} axiomatized the
geometry of projection distances in $\YY$.
With Sisto
\cite{BesBroFuj17} they showed that by a small change in the
projections and projection distances, a cleaner set of axioms is
satisfied---these will allow us to make an inductive
argument in the next section.
The following is \cite[Theorem~4.1]{BesBroFuj17} applied to $\YY$.
We
list here only those axioms that we will make use of and that are not
immediate from our particular definitions of $\YY$, $\pi_\Y$, and $d^\pi_\Y$.
A detailed verification that $\YY$ satisfies the hypotheses of
\cite[Theorem~4.1]{BesBroFuj17} can be found in \cite{ArzCasTao15}.
\begin{theorem}\label{BBFS}
  There exists $\theta \geqslant 0$ such that for each $\Y\in\YY$
  there is a projection $\pi'_\Y$ taking elements of $\YY$ to subsets
  of $\Y$ such that for all
  $\X\in\YY$ and $g\in G$ we have 
  $\pi'_\Y(\X)\subset B_\theta(\pi_\Y(\X))$ and
  $\pi'_{g\Y}(g\X)=g\pi'_\Y(\X)$.
  Furthermore, there are distance maps  $d_\Y(\X,\Z)=\diam
  \pi_\Y'(\X)\cup\pi_\Y'(\Z)$ with $|d_\Y-d^\pi_\Y|\leqslant 2\theta$ such
  that, for $\theta':=11\theta$, the following axioms are satisfied
  for all $\X,\,\Y,\,\Z,\,\W\in\YY$: 
  \begin{description}[labelwidth=3em, align=right]
    \item[(P 0)]$d^\pi_\Y(\X,\X)\leqslant \theta$ when $\X\neq\Y$.\label{P0}
    \item[(P 1)]If $d^\pi_\Y(\X,\Z)>\theta$ then 
      $d^\pi_\X(\Y,\Z)\leqslant\theta$ for all distinct $\X$, $\Y$, $\Z$.\label{P1}
\item[(SP 3)] If $d_\Y(\X,\Z)>\theta'$ then $d_\Z(\X,\W)=d_\Z(\Y,\W)$
  for all $\W\in\YY-\{\Z\}$.\label{SP3}
\item[(SP 4)] $d_\Y(\X,\X)\leqslant \theta'$ when $\X\neq\Y$.\label{SP4}
\end{description}
\end{theorem}

For more details on strongly contracting elements and many examples,
see \cite{ArzCasTao15}.

\begin{proposition}[{\cite[Lemma~2.2 and Proposition~2.3]{BesBroFuj17}}]\label{order}
  With $\theta'$ as in \fullref{BBFS}, for each  $\X$ and $\Z$ in
  $\YY$ define $\YY{(\X,\Z)}:=\{\Y\in\YY-\{\X,\Z\}\mid d_\Y(\X,\Z)>2\theta'\}$
  and $\YY{[\X,\Z]}:=\YY{(\X,\Z)}\cup\{\X,\Z\}$.
There is a total order $\order$ on
$\YY{[\X,\Z]}$ such if $\Y_0\order \Y_1\order \Y_2$ then
$d_{\Y_1}(\Y_0,\Y_2)=d_{\Y_1}(\X,\Z)$.
The relation $\Y_0\order \Y_1$ is defined by each of the
following equivalent conditions:
\setlength\multicolsep{3pt}
\begin{multicols}{2}
\begin{itemize}
\item $d_{\Y_0}(\X,\Y_1)>\theta'$
\item $d_{\Y_1}(\X,\Y_0)\leqslant \theta'$
\item $d_{\Y_1}(\Y_0,\Z)>\theta'$
\item $d_{\Y_0}(\Y_1,\Z)\leqslant \theta'$
\end{itemize}
\end{multicols}
\end{proposition}

\section{Embedding a tree's worth of copies of $[c]$.}

For a subset $H\subset G$, let $H^*:=H-\{1\}$, and consider $\hat{H}:=\bigcup_{k=1}^\infty (H^*)^k$.
We consider $\hat{H}$ to be a \emph{`tree's worth of copies of $H$'} in allusion to the case
of the free product $H*\mathbb{Z}/2\mathbb{Z}$ when $H$ is a group.
The group $H*\mathbb{Z}/2\mathbb{Z}$ acts on a tree with vertex
stabilizers conjugate to $H$, and every element that is not equal to $1$ or the generator
$z$ of $\mathbb{Z}/2\mathbb{Z}$ has a unique expression as 
$z^\alpha h_1zh_2z\cdots h_kz^\beta$  for some $k\in\mathbb{N}$, $\alpha,\,\beta\in\{0,1\}$, and
$h_i\in H^*$.

The na\"ive map $\hat{H}\to
X:(h_1,\dots,h_k)\mapsto h_1c\cdots h_kc.\bp$, where $c$ is a strongly
contracting element, is clearly not an injection for $H=[c]$, as it gives
collisions $(h^{-1},h)\mapsto h^{-1}chc.\bp\mapsfrom (h^{-1}ch)$.
To avoid collisions we remove a fraction of $[c]$ in four steps,
and use a slightly different map. The main technical result is:

\begin{proposition}\label{main}
Under the hypothesis of Theorem~\ref{maintheorem},
let $c$ be a strongly contracting element. 
After possibly passing to a power
  of $c$,
  there is a subset $G_4\subset [c]$
  that is divergent, has $\delta_{G_4}=\delta_G/2$, and for which the
  map $\hat G_4\to X:(g_1,\dots,g_k)\mapsto (\prod_{i=1}^kg_ic^2).\bp$
  is an injection.
\end{proposition}

The main theorem follows by an argument analogous to the one we used in \cite{ArzCasTao15}, which we reproduce for
the reader's convenience.

\begin{proof}[{Proof of \fullref{maintheorem}}]
 Let $c'\in G$ be a strongly contracting element for $G\act
 X$.
 Suppose that $N< E(c')$.
 Since $N$ is infinite, it has a finite index subgroup in common with
 $\langle c'\rangle$. But conjugation by an element of $G$ fixes $N$,
 so it moves $\langle c'\rangle$ by a bounded Hausdorff distance, which
 means $G=E(c')$ is virtually cyclic and $N$ is a finite index
 subgroup of $G$.
 However, $\langle c'\rangle$ has an undistorted orbit in
 $X$.
 Since this is a finite index subgroup of $G$, the growth of $G$ is
 only linear, contradicting the exponential growth hypothesis.
Thus, we may assume that $G$ is not virtually cyclic and that $N$
contains an element $g$ that is not in $E(c')$.
We showed in 
\cite[Proposition~3.1]{ArzCasTao15} that for sufficiently large $n$
the element $c:=g^{-1}(c')^{-n}g(c')^n$ is a strongly contracting element of
$N$. 

  Consider $G_4$ as
  provided by \fullref{main} with respect to $c$.
  Then $\hat G_4$ injects into $X$, and,
  moreover, the image is contained in $\langle\langle
  c\rangle\rangle.\bp\subset N.\bp$.
  Therefore, the growth rate of $N$ is at least as large as the growth
  rate of the image of $\hat G_4$, which we estimate using its
  Poincar\'e series:
  \begin{align*}
    \Theta_{\hat G_4}(s)&=\sum_{k=1}^\infty\sum_{(g_1,\dots,g_k)\in (G_4^*)^k}\exp(-s|g_1c^2\cdots g_kc^2|)\\
                        &\geqslant \sum_{k=1}^\infty\sum_{(g_1,\dots,g_k)\in (G_4^*)^k}\exp\left(-sk|c^2|-s\sum_{i=1}^{k}|g_i|\right)\\
    &=\sum_{k=1}^\infty\exp(-sk|c^2|)\sum_{(g_1,\dots,g_k)\in
      (G_4^*)^k}\prod_{i=1}^{k}\exp(-s|g_i|)\\
    &=\sum_{k=1}^\infty\exp(-sk|c^2|)\left(\sum_{g\in
      G_4^*}\exp(-s|g|)\right)^k\\
    &=\sum_{k=1}^\infty\left(\exp(-s|c^2|)\Theta_{G_4^*}(s)\right)^k\\
  \end{align*}
  Since $G_4$ is divergent, for sufficiently small positive $\epsilon$
  we have $\Theta_{G_4^*}(\delta_{G_4}+\epsilon)\geqslant
  \exp((\delta_{G_4}+\epsilon)|c^2|)$, so $\Theta_{\hat G_4}(\delta_{G_4}+\epsilon)$
  diverges, which implies $\delta_{\hat G_4}\geqslant
  \delta_{G_4}+\epsilon$. Thus, $\delta_N\geqslant \delta_{\hat G_4}\geqslant \delta_{G_4}+\epsilon>\delta_{G_4}=\delta_G/2$.
\end{proof}

The remainder of this section is devoted to the construction of the set $G_4$
satisfying the conclusion of \fullref{main}.
Here is a brief overview.
We need a subset of $[c]$ such that the given map is an injection.
It would be preferable if we could take conjugates of $c$ by elements
$g$ that have no long projection to any element of $\YY$.
It is easy to build an injection based on such elements, but,
unfortunately, there are too few of them in our setting---the growth rate of the set
of such elements is strictly smaller than $\delta_G$, so the growth
rate of $c$--conjugates by such elements is strictly smaller than
$\delta_G/2$.
Instead, we consider elements $g$ that do not have long projections to
$\E$ and $g\E$; in a sense, these are elements `orthogonal to $\YY$
at their endpoints', rather than `orthogonal to $\YY$' throughout.
The desired condition can be achieved with a small modification near
the ends of $g$, so this does not change the growth rate.
We call this set of elements $G_1$ and the conjugates of (a power
of) $c$ by these elements $G_2$.
We define $G_3$ by passing to a maximal subset of $G_2$ such that
elements are sufficiently far apart.
This does not change the set much; in particular, the growth rate is
unchanged.
However, it will be an important point for the injection argument, because we
show in \fullref{sameaxissameelement} that if $g$ and
$h$ are in $G_3$ then $g\E=h\E$ implies $g=h$.
The final refinement is to pass to the subset $G_4$ of $G_3$ of elements
that are not `in the shadow' of some other element of $G_3$, that is to
say, elements $g$ such that there does not exist $h$ such that a
geodesic from $\bp$ to $g.\bp$ passes close to $h.\bp$. 
The crux of the argument, \fullref{lemma:divrate}, is to show that at least half of $G_3$ is
unshadowed, so $G_4$ is divergent with growth rate $\delta_G/2$.
Finally, in \fullref{lemma:injection}, we check that $G_4$ gives the
desired injection.

\medskip

Fix an element $f_0\in G$ such that $f_0\E$
is disjoint from $\E$, $\bp\in\pi_\E(f_0.\bp)$, and $f_0.\bp\in\pi_{f_0\E}(\bp)$.
To see that such an element exists, first note that there exists
$g\in G-E(c)$, for instance, as in the first paragraph of the proof of
\fullref{maintheorem}.
If $\E$ and $g\E$ are disjoint, let $f_1$ and $f_2$ be elements of $G$ such that
$f_1.\bp\in\E$ and $f_2.\bp\in g\E$ realize the minimum distance
between $\E$ and $g\E$.
Then the element $f_0:=f_1^{-1}f_2$ satisfies our requirements.
If $g\E$ and $\E$ are not disjoint consider $g\E$ and $c^ng\E$, for
some $n$.
If they intersect then, by \pzero:
\[2\theta\geqslant d^\pi_\E(g\E,g\E)+d^\pi_\E(c^ng\E,c^ng\E)\geqslant
  d^\pi_\E(g\E,c^ng\E)\geqslant |c^n|\]
This is impossible once $n$ is sufficiently large as $c$ is strongly contracting. So, $g\E$ and
$c^ng\E$ are disjoint for such $n$, and we
get $f_0$ by the previous argument after replacing $g$ with  $g^{-1}c^ng$.

Since $\E$ and $f_0\E$ are disjoint and $\bp$ and $f_0.\bp$ are
contained in one another's projections, strong contraction of $c$, and hence of $\E$, gives a constant $C\geqslant 0$ such that:
\begin{equation}
  \label{eq:3}
  d^\pi_{f_0\E}(\bp,f_0.\bp)=\diam \pi_{f_0\E}(\bp)\leqslant C \quad\text{
    and }\quad d^\pi_\E(\bp,f_0.\bp)=\diam\pi_\E(f_0.\bp)\leqslant C
\end{equation}

In the sequel, we use the following notation: $|f_0|$ is the length of
the element $f_0$ just defined;
  $\Delta$ is as in the definition of purely exponential growth
of $G$;
$C$ is a contraction constant for $\E$; $C'$ is the corresponding
constant from \fullref{BGI}; $\theta$ and $\theta'$ are as in
\fullref{BBFS}; $K$ is a fixed constant strictly greater than $\max\{C,\theta+\theta'/2\}$.
We call these, collectively, `the constants'.
The terms `small' and `close' mean bounded by some combination
of the constants.
When possible we decline to compute these explicitly since only
finitely many such combinations appear in the proof, except where
noted.
Furthermore, $\Delta$ depends only on $G$, and the others depend only
on $\E=E(c).\bp$.
Since $E(c)=E(c^p)$ for all $p\neq 0$, we can, and will, pass to high powers of
$c$ to make $|c^p|$ much larger than all
of the constants and combinations of them that we encounter.

\medskip

Set $G_1:=\{g\in G \mid d^\pi_{\E}(\bp,g.\bp)\leqslant 2K \text{ and
} d^\pi_{g\E}(\bp,g.\bp)\leqslant 2K \text{ and }g\E\neq\E\}$.
This is a subset of $G$ that is closed under taking inverses.
\begin{lemma}
 For every $g\in G$ at least one of the elements $g$,
$f_0g$, $gf_0$, or $f_0gf_0$ belongs to $G_1$.
\end{lemma}
\begin{proof}
First, consider $g\notin E(c)$ with $|g|\leqslant K$.
Recall $g\in E(c)$ if and only if $g\E=\E$. 
By definition, $\pi_\E(g.\bp)$ is the set of points of $\E$ minimizing the distance
to $g.\bp$. 
By hypothesis, $\bp$ is a point of $\E$ at distance at most $K$ from
$g.\bp$ so
$d(g.\bp,\pi_\E(g.\bp))\leqslant K$, and $d^\pi_\E(\bp,g.\bp)=\diam
\,\{\bp\}\cup\pi_\E(g.\bp)\leqslant 2K$. 
The same argument for $\bp$ projecting to $g\E$ gives
$d^\pi_{g\E}(\bp,g.\bp)\leqslant 2K$.
Thus, elements $g$ of this form already belong to $G_1$.

Next, consider an element $g\in E(c)$ such that
$|g|\leqslant K$.
Since $g\in E(c)$, we have $f_0g\E=f_0\E\neq\E$ and
$\pi_{\E}(g.\bp)=g.\bp$, so
$d^\pi_\E(\bp,g.\bp)=d(\bp,g.\bp)\leqslant K$.
Using this estimate and \eqref{eq:3}, we see:
\[d^\pi_{f_0g\E}(\bp,f_0g.\bp)\leqslant
  d^\pi_{f_0g\E}(\bp,f_0.\bp)+d^\pi_{f_0g\E}(f_0.\bp,f_0g.\bp)=d^\pi_{f_0\E}(\bp,f_0.\bp)+d^\pi_{\E}(\bp,g.\bp)\leqslant
  C+K< 2K\]
In the other direction, using the fact that
$\bp\in\pi_\E(f_0.\bp)\subset\pi_\E(f_0\E)$, along with \pzero:
\[d^\pi_\E(\bp,f_0g.\bp)\leqslant d^\pi_\E(\bp,f_0\E)\leqslant
  d^\pi_\E(f_0\E,f_0\E)\leqslant \theta<K\]
Note that we did not use $d^\pi_\E(\bp,g.\bp)\leqslant K$ for this
direction---the inequality is valid for any $g\in E(c)$.

Suppose $g\notin E(c)$ and $d^\pi_{\E}(\bp,g.\bp)> K$ then:
\[\theta<K<d_\E^\pi(\bp,g.\bp)=d^\pi_{f_0\E}(f_0.\bp,f_0g.\bp)\leqslant
  d^\pi_{f_0\E}(\E,f_0g\E)\]
This contradicts \pzero if $\E= f_0g\E$, since, by hypothesis,
$f_0\E\neq\E$ and $f_0g\E\neq f_0\E$.
Thus,  $\E$, $f_0\E$, and $f_0g\E$ are distinct, and we can apply
\pone to get:
\[d^\pi_\E(\bp,f_0g.\bp)\leqslant d^\pi_\E(f_0\E,f_0g\E)\leqslant\theta<K\]

For $|g|\leqslant K$ we are done, either $g$ or $f_0g$ is in $G_1$, and for $|g|>K$ we have shown
that there is at least one choice of $g'\in\{g,f_0g\}$
such that $g'\E\neq\E$ and  $d^\pi_\E(\bp,g'\!.\bp)\leqslant K$.
If $d^\pi_{g'\E}(\bp,g'\!.\bp)\leqslant K$ then we are done, so
suppose not.
Consider the possibility that $g'f_0\E=\E$. Then $g'f_0.\bp\in\E$, so
$\bp\in\pi_\E(f_0.\bp)$ implies
$g'\!.\bp\in\pi_{g'\E}(g'f_0.\bp)\subset \pi_{g'\E}(\E)$.
Since $g'\E\neq \E$, \pzero says $d^\pi_{g'\E}(\E,\E)\leqslant \theta$, so:
\[K<d^\pi_{g'\E}(\bp,g'\!.\bp)\leqslant d^\pi_{g'\E}(\E,\E)\leqslant
  \theta<K\]
This is a contradiction, so $\E$, $g'\E$, and $g'f_0\E$ are
distinct. Observe, since $g'\!.\bp\in\pi_{g'\E}(g'f_0.\bp)$:
\[d^\pi_{g'\E}(\E,g'f_0\E)\geqslant
  d^\pi_{g'\E}(\bp,g'f_0.\bp)\geqslant d^\pi_{g'\E}(\bp,g'\!.\bp)> K>\theta\]
Thus, by \pone and the fact that
$g'f_0.\bp\in\pi_{g'f_0\E}(g'\!.\bp)$, we have 
$d^\pi_{g'f_0\E}(\bp,g'f_0.\bp) \leqslant
d^\pi_{g'f_0\E}(\E,g'\E)\leqslant \theta<K$.

To check that the first inequality has not been spoiled, use the fact
that $d^\pi_{g'\E}(\E,g'f_0\E)>\theta$, so \pone implies
$d^\pi_\E(g'\E,g'f_0\E)\leqslant \theta$, which gives:
\[d_\E^\pi(\bp,g'f_0.\bp)\leqslant
  d^\pi_\E(\bp,g'\!.\bp)+d^\pi_\E(g'\!.\bp,g'f_0.\bp)\leqslant
  K+d^\pi_\E(g'\E,g'f_0\E)<K+\theta<2K\qedhere\]
\end{proof}

Define $\phi_0\from G\to G_1$ by fixing $G_1$ and sending an element
$g\in G-G_1$ to an arbitrary element of  the nonempty set $\{f_0g,\,gf_0,\,f_0gf_0\}\cap G_1$.
The map $\phi_0$ is surjective, at most 4-to-1, and changes norm by
at most $2|f_0|$.

\medskip

For each $p\in\mathbb{N}$, define $G_{2,p}:=\{g^{-1}c^pg \mid g\in G_1 \}$ and
$\phi_{1,p}\from G_1\to G_{2,p}:g\mapsto g^{-1}c^pg$.
\begin{lemma}\label{phi1estimate} 
  If $p$ is sufficiently large then for every $g\in G_1$ we have:
  \[2|g|+|c^p|-8C'-8K\leqslant |\phi_{1,p}(g)|\leqslant
2|g|+|c^p|\]
\end{lemma}
\begin{proof}
  The upper bound is clear.
 We derive a lower bound from strong contraction.
From the definition of $G_1$ it follows that
$\pi_{g^{-1}\E}(\bp)\subset \bar B_{2K}(g^{-1}\!.\bp)$ and
$\pi_{g^{-1}\E}(g^{-1}c^pg.\bp)\subset \bar B_{2K}(g^{-1}c^p\!.\bp)$, so:
\begin{equation}
  \label{eq:1}
|c^p|-4K\leqslant  d^\pi_{g^{-1}\E}(\bp,g^{-1}c^pg.\bp)\leqslant |c^p|+4K
\end{equation}
Let $\gamma$ be a geodesic from $\bp$ to $g^{-1}c^pg.\bp$.
Its endpoints have projection to $g^{-1}\E$ at distance at least
$|c^p|-4K\gg C'$ from one another, for $p$ sufficiently large, as $c$ is strongly contracting.
Thus, for $t_0$ and $t_1$ as in \fullref{BGI}, we have
$d(\gamma(t_0),\pi_{g^{-1}\E}(\bp))\leqslant 2C'$, so
$d(\gamma(t_0),g^{-1}\!.\bp)\leqslant 2C'+2K$, and, similarly,
$d(\gamma(t_1),g^{-1}c^p\!.\bp)\leqslant 2C'+2K$.
\begin{align*}
  |\phi_{1,p}(g)|&=|\gamma|=d(\bp,\gamma(t_0))+d(\gamma(t_0),\gamma(t_1))+d(\gamma(t_1),g^{-1}c^pg.\bp)\\
  &\geqslant   \left(d(\bp,g^{-1}\!.\bp)-(2C'+2K)\right) +
    \left(d(g^{-1}\!.\bp,g^{-1}c^p\!.\bp)-2(2C'+2K)\right)\\
  &\qquad\qquad+\left(d(g^{-1}c^p\!.\bp,g^{-1}c^pg.\bp)-(2C'+2K)\right)\\
                 &=2|g|+|c^p|-8C'-8K\qedhere
\end{align*}
\end{proof}

The following lemma also follows from \eqref{eq:1}.
\begin{lemma}\label{G2order}
  Let $g^{-1}c^pg=\phi_{1,p}(g)\in G_{2,p}$. If $p$ is sufficiently
  large then $\E\order g^{-1}\E\order g^{-1}c^pg\E$ for the order
  $\order$ on $\YY[\E,g^{-1}c^pg\E]$ of \fullref{order}.
\end{lemma}

We also claim $\phi_{1,p}$ is bounded-to-one, independent of $p$.
To see this, fix $g\in G_1$ and
consider $h\in G_1$ such that $\phi_{1,p}(g)=\phi_{1,p}(h)$.
Then $gh^{-1}$ commutes with $c^p$, so
$gh^{-1}\in E(c^p)=E(c)$.
Thus:
\[|gh^{-1}|=d^\pi_\E(\bp,gh^{-1}\!.\bp)\leqslant
d^\pi_\E(\bp,g.\bp)+d^\pi_\E(g.\bp,gh^{-1}\!.\bp)=
d^\pi_\E(\bp,g.\bp)+d^\pi_{hg^{-1}\E}(h.\bp,\bp)=d^\pi_\E(\bp,g.\bp)+d^\pi_{\E}(h.\bp,\bp)\leqslant
4K\]
So, $h$ satisfies $h^{-1}\!.\bp\in \bar B_{4K}(g^{-1}\!.\bp)$.
By properness of $G\act X$, $\#G.\bp\cap\bar
B_{4K}(g^{-1}\!.\bp)=\#G.\bp\cap\bar B_{4K}(\bp)$ is finite.

Let $G_{3,p}$ be a maximal $(6K+1)$--separated subset of $G_{2,p}$,
that is, a subset that is maximal for inclusion among those with the property that
$d(g.\bp,h.\bp)\geqslant 6K+1$ for distinct elements $g$ and $h$.
Let $\phi_{2,p}\from G_{2,p}\to G_{3,p}$ be a choice of closest point.
This map is surjective.
By maximality, $\phi_{2,p}$ moves points a distance less than $6K+1$.
Thus, by properness of $G\act X$, the map $\phi_{2,p}$ is
bounded-to-one, independent of $p$.

\begin{lemma}\label{sameaxissameelement}
If $p$ is sufficiently large then $g^{-1}c^pg\E=h^{-1}c^ph\E$ for $g^{-1}c^pg$ and
$h^{-1}c^ph$ in $G_{3,p}$ implies  $g^{-1}c^pg=h^{-1}c^ph$.
\end{lemma}
\begin{proof}
Since $g\in G_1$, $d^\pi_{g\E}(\bp,g.\bp)\leqslant 2K$, and: 
\[d^\pi_{g^{-1}c^pg\E}(\bp,g^{-1}c^pg.\bp) \leqslant 
                                          d^\pi_{g^{-1}c^pg\E}(\bp,g^{-1}c^p\!.\bp)+d^\pi_{g^{-1}c^pg\E}(g^{-1}c^p\!.\bp,g^{-1}c^pg.\bp)\leqslant d^\pi_{g^{-1}c^pg\E}(\E,g^{-1}\E)+2K\]
Furthermore, $g\in G_1$ implies $\E\neq g^{-1}\E\neq g^{-1}c^pg\E$.
By \eqref{eq:1}, 
$d^\pi_{g^{-1}\E}(\E,g^{-1}c^pg\E)\geqslant |c^p|-4K\gg\theta$, so by
\pzero, $\E\neq g^{-1}c^pg\E$.
Thus $\E$, $g^{-1}\E$, and $g^{-1}c^pg\E$ are distinct  and
we can apply \pone to see $d^\pi_{g^{-1}c^pg\E}(\E,g^{-1}\E)\leqslant \theta<K$.
Plugging this into previous inequality gives:
\begin{equation}
  \label{eq:6}
d^\pi_{g^{-1}c^pg\E}(\bp,g^{-1}c^pg.\bp)<3K  
\end{equation}

The same computation applies for $h$, so $\pi_{g^{-1}c^pg\E}(\bp)\subset \bar
B_{3K}(g^{-1}c^pg.\bp)\cap\bar B_{3K}(h^{-1}c^ph.\bp)$.
Thus,  $g^{-1}c^pg$ and $h^{-1}c^ph$ are elements at distance at most
$6K$ in a $(6K+1)$--separated set; hence, they are equal.
\end{proof}

For each $D\geqslant 0$, consider the set $G_{4,p,D}'$ consisting of elements $g^{-1}c^pg\in G_{3,p}$ such that
there exists a different element
$h^{-1}c^ph\in G_{3,p}$ such that $h^{-1}c^phc^{2p}\!.\bp$ is
within distance $D$ 
of a geodesic $\gamma$ from
$\bp$ to $g^{-1}c^pg.\bp$.
Define $G_{4,p,D}:=G_{3,p}-G_{4,p,D}'$. 
\begin{lemma}\label{lemma:divrate}
For all $D\geqslant 0$, for $p$ sufficiently large, $G_{4,p,D}$ is divergent and $\delta_{G_{4,p,D}}=\delta_G/2$.
\end{lemma}
\begin{proof}
The maps $\phi_{2,p}$, $\phi_{1,p}$, and $\phi_0$ are
surjective and 
bounded-to-one, with bound independent of $p$, so their composition is
as well.
Furthermore, we know how they change norm: $\phi_0$ moves points at most $2|f_0|$, $\phi_{2,p}$
  moves less than $6K+1$, and $|\phi_{1,p}(g)|$ is estimated in \fullref{phi1estimate}. 
Putting these together, for any $r\geqslant 0$ and  $g\in G\cap
S_r^\Delta$ we have:
\begin{equation}
  \label{eq:9}
2r+|c^p|-4|f_0|-8C'-14K-1  \leqslant |\phi_{2,p}\circ\phi_{1,p}\circ\phi_0(g)|<2r+|c^p|+2\Delta+4|f_0|+6K+1
\end{equation}
Let $t:=2r+|c^p|-4|f_0|-8C'-14K-1$, $E:=4|f_0|+4C'+10K+1$, and $\Delta':=2(\Delta+E)$,
so that \eqref{eq:9} shows:
\[\phi_{2,p}\circ\phi_{1,p}\circ\phi_0(G\cap S_r^\Delta)\subset
  G_{3,p}\cap S_{t}^{\Delta'}\subset \phi_{2,p}\circ\phi_{1,p}\circ\phi_0(G\cap S_{r-E}^{\Delta+2E})\]
This lets us compare the size of spherical shells in $G_{3,p}$ and $G$:
\begin{equation}
  \label{eq:11}
  \#G\cap S_{r-E}^{\Delta+2E}\geqslant \#G_{3,p}\cap S_t^{\Delta'}\gmul \#G\cap S_r^\Delta
\end{equation}
Pure exponential growth of $G$ says that $\#G\cap S_r^\Delta\emul
\exp(r\delta_G)$. Combining this with \eqref{eq:11}, we have:
\begin{equation}
  \label{eq:10}
  \#G_{3,p}\cap S_t^{\Delta'}\emul \exp(\delta_Gr)\emul\exp(-\delta_G|c^p|/2)\exp(t\delta_G/2)
\end{equation}

This tells us that $\delta_{G_{3,p}}=\delta_G/2$ and $G_{3,p}$ is
divergent.

Now we will estimate an upper bound for $\#G_{4,p,D}'\cap
S_r^{\Delta'}$ and see that for large $p$ and $r$ it is less than half of $\#G_{3,p}\cap
S_r^{\Delta'}$.
Thus, to get  $G_{4,p,D}$ we threw away less than half of $G_{3,p}$,
at least outside a sufficiently large radius.
We conclude that $\delta_{G_{4,p,D}}=\delta_G/2$ and $G_{4,p,D}$ is
divergent.

\medskip

Consider $g^{-1}c^pg\in G_{4,p,D}'\cap S_r^{{\Delta'}}$ for any $r> 7|c^p|$.
By definition of $G_{4,p,D}'$, there exists 
$h^{-1}c^ph\in G_{3,p}$ such that $h^{-1}c^ph\not=g^{-1}c^pg$ and $h^{-1}c^phc^{2p}\!.\bp$ is
close to a geodesic $\gamma$ from
$\bp$ to $g^{-1}c^pg.\bp$.

Let $\order$ be the order of \fullref{order} on $\YY[
\E,g^{-1}c^pg\E]$.
The first step of the proof is to show that $\E$, $g^{-1}\E$, 
$g^{-1}c^pg\E$, $h^{-1}\E$, and $h^{-1}c^ph\E$ are distinct elements
of $\YY[\E,g^{-1}c^pg\E]$, and that the ordering is one of the two possibilities shown in
\fullref{fig:3_3_1} and \fullref{fig:3_3_2}.
\begin{figure}[h]
  \centering
  \labellist
  \tiny
  \pinlabel $g^{-1}c^pg.\bp$ [tr] at 477 32
  \pinlabel $h^{-1}\!.\bp$ [t] at 45 2
  \pinlabel $h^{-1}c^p\!.\bp$ [t] at 80 2
  \pinlabel $g^{-1}\!.\bp$ [tl] at 280 30
  \pinlabel $g^{-1}c^p\!.\bp$ [tl] at 315 32
  \pinlabel $h^{-1}c^ph.\bp$ [t] at 141 7
  \pinlabel $h^{-1}c^phc^{2p}\!.\bp$ [t] at 187 16
  \small
  \pinlabel $\gamma$ [t] at 419 49
  \pinlabel $a$ [b] at 200 55
  \pinlabel $g$ [b] at 142 88
  \pinlabel $g$ [b] at 351 88
  \pinlabel $h$ [t] at 28 23
  \pinlabel $h$ [t] at 101 23
  \endlabellist
  \includegraphics[width=\textwidth]{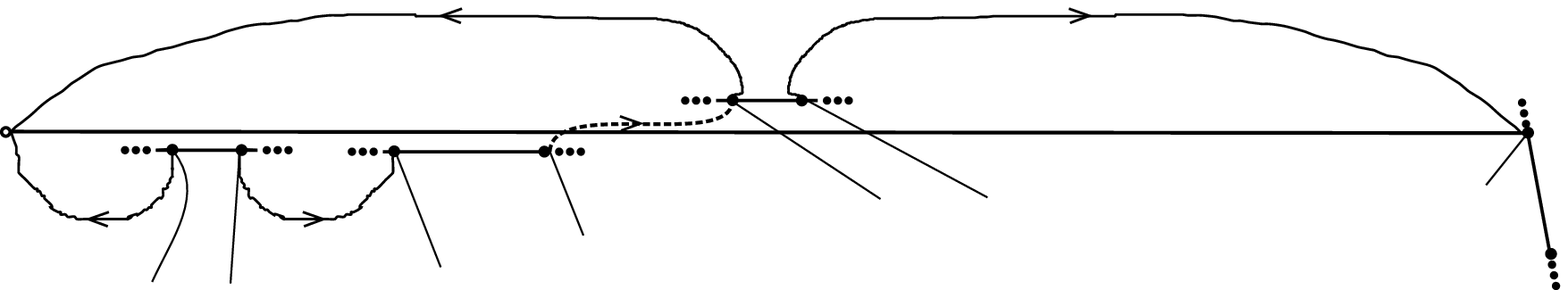}
  \caption{$h^{-1}c^ph\E$ before $g^{-1}\E$, that is, $h^{-1}c^ph\E\order g^{-1}\E$}\label{fig:3_3_1}
\end{figure}

\begin{figure}[h]
  \centering
  \labellist
  \tiny
  \pinlabel $g^{-1}c^pg.\bp$ [tr] at 477 32
  \pinlabel $h^{-1}\!.\bp$ [tr] at 135 4
  \pinlabel $h^{-1}c^p\!.\bp$ [tl] at 150 4
  \pinlabel $g^{-1}\!.\bp$ [t] at 214 4
  \pinlabel $g^{-1}c^p\!.\bp$ [tl] at 244 7
  \pinlabel $h^{-1}c^ph.\bp$ [t] at 300 26 
  \pinlabel $h^{-1}c^phc^{2p}\!.\bp$ [tl] at 343 26
\small
  \pinlabel $\gamma$ [t] at 419 49
  \pinlabel $b$ [br] at 213 38
  \pinlabel $b'$ [r] at 258 40
  \pinlabel $g$ [b] at 142 88
  \pinlabel $g$ [b] at 351 88
  \pinlabel $h$ [t] at 55 26
  \pinlabel $h$ [t] at 233 26
  \endlabellist
    \includegraphics[width=\textwidth]{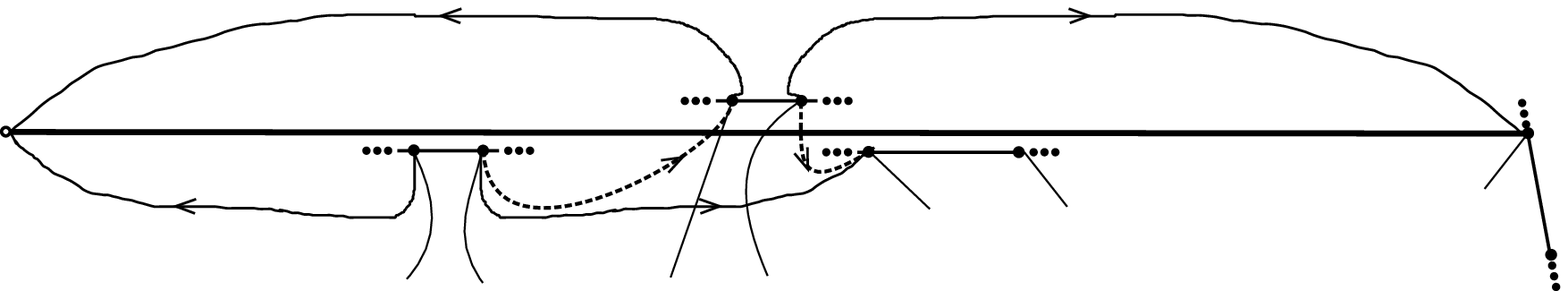}
  \caption{$h^{-1}c^ph\E$ after $g^{-1}\E$, that is, $g^{-1}\E\order h^{-1}c^ph\E$}
  \label{fig:3_3_2}
\end{figure}

By \fullref{G2order}, $\E\order g^{-1}\E\order g^{-1}c^pg\E$, so these
three are distinct.
Similarly, $\E$, $h^{-1}\E$, and $h^{-1}c^ph\E$ are
distinct.
\fullref{sameaxissameelement} implies $g^{-1}c^pg\E\neq h^{-1}c^ph\E$.

We have $|c^p|+2|g|\geqslant |g^{-1}c^pg|\gadd |h^{-1}c^phc^{2p}|$
since
$h^{-1}c^phc^{2p}\!.\bp$ is close to a geodesic from $\bp$ to
$g^{-1}c^pg.\bp$.
On the other hand, any geodesic  from $\bp$ to $h^{-1}c^phc^{2p}\!.\bp$
has projection to $h^{-1}c^ph\E$ of diameter greater than $|c^{2p}|-3K$ by
\eqref{eq:6}.
This is much larger than $C'$ when $p$ is large, so
$|h^{-1}c^phc^{2p}|\eadd |h^{-1}c^ph|+|c^{2p}|\gadd 3|c^p|+2|h|$ by
\fullref{BGI} and
\fullref{phi1estimate}.
Thus:
\begin{equation}
  \label{eq:8}
|g|\gadd|h|+|c^p|  
\end{equation}

However, by definition of $G_1$, if $h^{-1}\E=g^{-1}\E$, then:
\[4K\geqslant d^\pi_{g^{-1}\E}(\bp,g^{-1}\!.\bp)+d^\pi_{h^{-1}\E}(\bp,h^{-1}\!.\bp)\geqslant
  d(g^{-1}\!.\bp,h^{-1}\!.\bp)\geqslant |g|-|h|\gadd |c^p|\]
This is a contradiction for sufficiently large $p$.
Similar considerations show $h^{-1}\E\neq g^{-1}c^pg\E$, since $\bp$
projects close to $h^{-1}\!.\bp$ in $h^{-1}\E$, by definition of $G_1$, and
close to $g^{-1}c^pg.\bp$ in
$g^{-1}c^pg\E$, by \eqref{eq:6},  but
$|h|  \ll |g^{-1}c^pg|$, by \fullref{phi1estimate} and~(\ref{eq:8}).

Next we show that $h^{-1}\E$ and $h^{-1}c^ph\E$ belong to
$\YY[\E,g^{-1}c^pg\E]$, and
in the course of the proof we will observe $g^{-1}\E\neq h^{-1}c^ph\E$.
By hypothesis, there exists $t$ such that $d(\gamma(t),h^{-1}c^phc^{2p}\!.\bp)\leqslant D$.
This implies
$d^\pi_{h^{-1}c^ph\E}(\gamma(t),h^{-1}c^phc^{2p}\!.\bp)\leqslant 2D$.
Since $d^\pi_{h^{-1}c^ph\E}(\bp,h^{-1}c^ph.\bp)<3K$, by
\eqref{eq:6}, we have $d^\pi_{h^{-1}c^ph\E}(\bp,\gamma(t))\geqslant
|c^{2p}|-2D-3K$, which is large for $p$ sufficiently large.
Let $t_0$ and $t_1$ be the first and last times $\gamma$ is distance
$C'$ from $h^{-1}c^ph\E$, as in \fullref{BGI} with respect to $h^{-1}c^ph\E$.
We cannot have $t\leqslant t_0$, since then
$d^\pi_{h^{-1}c^ph\E}(\bp,\gamma(t))\leqslant C'$, but $\pi_{h^{-1}c^ph\E}(\bp)$ is close
  to $h^{-1}c^ph.\bp$, hence, far from $\gamma(t)$ and $h^{-1}c^phc^{2p}\!.\bp$.

  If $t\geqslant t_1$ then
  $d^\pi_{h^{-1}c^ph\E}(\gamma(t),g^{-1}c^pg.\bp)\leqslant C'$, so:
  \begin{align*}
    d^\pi_{h^{-1}c^ph\E}(\bp,g^{-1}c^pg.\bp)&\geqslant
                                              d^\pi_{h^{-1}c^ph\E}(\bp,\gamma(t))-d^\pi_{h^{-1}c^ph\E}(\gamma(t),g^{-1}c^pg.\bp)\\
    &\geqslant |c^{2p}|-3K-2D-C'
  \end{align*}

  If $t_0<t<t_1$ then we use \fullref{BGI} to say $d^\pi_{h^{-1}c^ph\E}(\bp,g^{-1}c^pg.\bp)\geqslant
                                                  |\gamma(t_0,t_1)|-4C'$, and
                                                  then estimate:
  \begin{align*}
                                                  |\gamma(t_0,t_1)|
                                                &\geqslant
                                                  d(\gamma(t_0),\gamma(t))\\
    &\geqslant d\left(\pi_{h^{-1}c^ph\E}(\gamma(t_0)),\pi_{h^{-1}c^ph\E}(\gamma(t))\right)-C'-D\\
                                                &\geqslant
                                                  d^\pi_{h^{-1}c^ph\E}(\gamma(t_0),\gamma(t))-\diam
                                                  \pi_{h^{-1}c^ph\E}(\gamma(t_0))-\diam\pi_{h^{-1}c^ph\E}(\gamma(t))
                                                  -C'-D\\
  &\geqslant
                                                  d^\pi_{h^{-1}c^ph\E}(\gamma(t_0),\gamma(t))-2C-C'-D\\
    &\geqslant d^\pi_{h^{-1}c^ph\E}(\bp,\gamma(t))-d^\pi_{h^{-1}c^ph\E}(\gamma(t_0),\bp)-2C-C'-D\\
    &\geqslant |c^{2p}|-2D-3K-C'-2C-C'-D
  \end{align*}
  Thus, $h^{-1}c^ph\E\in\YY[\E,g^{-1}c^pg\E]$ once $p$ is sufficiently large.
Additionally, this shows $g^{-1}\E\neq h^{-1}c^ph\E$ because, by
\eqref{eq:1}, \pzero,  and the fact from \fullref{BBFS} that
$|d_\Y-d^\pi_\Y|\leqslant 2\theta$, we have
$d_{g^{-1}\E}(\E,g^{-1}c^pg\E)\eadd |c^p|$, while the estimates above show $d_{h^{-1}c^ph\E}(\E,g^{-1}c^pg\E)\gadd
|c^{2p}|$, and these are incompatible for sufficiently large $p$.
Thus, the five axes are distinct.

  Now consider $h^{-1}\E$.
  Recall, $\gamma(t)$ is close to
  $h^{-1}c^phc^{2p}\!.\bp$, but $h^{-1}c^phc^{2p}\!.\bp$ is far from
  $h^{-1}\E$, so strong contraction implies
  $d^\pi_{h^{-1}\E}(\gamma(t), h^{-1}c^phc^{2p}\!.\bp)\leqslant C$.
 Since $\bp$ projects close to $h^{-1}\!.\bp$ in $h^{-1}\E$ and
 $h^{-1}c^phc^{2p}\!.\bp\in h^{-1}c^ph\E$ projects close to
 $h^{-1}c^p\!.\bp$, \fullref{BGI} says $\gamma$ must pass close to $h^{-1}c^p\!.\bp$.
Now we can run the same argument as for $h^{-1}c^ph\E$ to see
$h^{-1}\E\in\YY[\E,g^{-1}c^pg\E]$ once $p$ is sufficiently large.

The first step of the proof is completed by observing that
$g^{-1}\E\order h^{-1}\E$ implies $|h|\eadd |g|+|c^p|$, which cannot be true
when $p$ is sufficiently large, by \eqref{eq:8}.
Thus, $h^{-1}\E$ comes before $g^{-1}\E$ and $h^{-1}c^ph\E$ under
$\order$, and we are left with the possibilities that $h^{-1}c^ph\E\order g^{-1}\E$, as in  \fullref{fig:3_3_1}, or the converse, as in \fullref{fig:3_3_2}.

In the case of \fullref{fig:3_3_1}, we have $h^{-1}c^ph\E\order g^{-1}\E$,
so the projection of $h^{-1}c^phc^{2p}\!.\bp$ to $g^{-1}\E$ is close
to the projection of $\bp$, which we know to be close to $g^{-1}\!.\bp$.
Write
$g^{-1}\!.\bp=h^{-1}c^phc^{2p}a.\bp$ as in
\fullref{fig:3_3_1} with
$|g|\eadd 2|h|+3|c^p|+|a|$.

In the case of \fullref{fig:3_3_2}, we have
$h^{-1}\E\order g^{-1}\E$ and $g^{-1}\E\order h^{-1}c^ph\E$.
The former implies the projection of $h^{-1}c^p\!.\bp$ to
$g^{-1}\E$ is close to the projection of $\bp$, which we know to be
close to $g^{-1}\!.\bp$, while the latter implies the projection of
$h^{-1}c^ph.\bp$ to $g^{-1}\E$ is close to the projection of
$g^{-1}c^pg.\bp$, which we know to be close to $g^{-1}c^p\!.\bp$.
Write
$g^{-1}\!.\bp=h^{-1}c^pb.\bp$ with $|g|\eadd |h|+|c^p|+|b|$ and write
$h.\bp=bc^pb'.\bp$ as in \fullref{fig:3_3_2} with $|h|\eadd |b|+|c^p|+|b'|$; together these give
$|g|\eadd 2|b|+2|c^p|+|b'|$. 

\medskip

Suppose we are in the case of \fullref{fig:3_3_2}, so there
are elements $b$ and $b'$ such that $(r-|c^p|)/2\eadd |g|\eadd
2|b|+2|c^p|+|b'|$.
Since $G$ has purely exponential growth,
if $i\leqslant |b|<i+1$ there are, up to a bounded multiplicative error
independent of $p$, $r$, and $i$,
at most $\exp(\delta_Gi)$ possible choices for $b$ and
at most $\exp(\delta_G(\frac{r-5|c^p|}{2}-2i))$ choices of $b'$, so there is an
upper bound for the number of possible elements $g$ by a multiple of:
\begin{equation}
  \label{eq:2}
  \sum_{i=0}^{\frac{r-5|c^p|}{4}}\exp(\delta_Gi)\exp\left(\delta_G\left(\frac{r-5|c^p|}{2}-2i\right)\right)<\frac{\exp(r\delta_G/2)}{\exp\left(5\delta_G|c^p|/2\right)(1-\exp(-\delta_G))}
\end{equation}
The case of \fullref{fig:3_3_1} is similar, but gives
an even smaller upper bound\footnote{Replace each `5' in \eqref{eq:2}
  with a `7'. This accounts for
the restriction that $r-7|c^p|>0$.}.
Thus, for all sufficiently large $p$ and $r$:
\begin{equation}
  \label{eq:12}
  \# G_{4,p,D}'\cap
S_r^{\Delta'}\lmul \exp(-5\delta_G|c^p|/2)\exp(r\delta_G/2)
\end{equation}

Combining \eqref{eq:10} and \eqref{eq:12} gives:

\begin{equation}
  \label{eq:7}
  \# G_{4,p,D}'\cap S_r^{\Delta'}\lmul \exp(-2|c^p|\delta_G)\cdot\#G_{3,p}\cap S_r^{\Delta'}
\end{equation}
Crucially, the multiplicative constant in this asymptotic inequality
does not depend on $p$, so for $p$ sufficiently large, $\exp(2|c^p|\delta_G)$ is more than twice the
multiplicative constant, and  \eqref{eq:7} becomes a true inequality
$\#G'_{4,p,D}\cap S_r^{\Delta'}< \frac{1}{2}\#G_{3,p}\cap S_r^{\Delta'}$.
We conclude that to get $G_{4,p,D}$ from $G_{3,p}$ we threw away fewer than half
of the points of $G_{3,p}$ in each spherical shell $S_r^{\Delta'}$ such
that $r>7|c^p|$.
\end{proof}

\begin{lemma}\label{lemma:injection}
  For all sufficiently large $D$, for all sufficiently large $p$, the map $\hat{G}_{4,p,D}\to X:(g_1,\dots,g_k)\mapsto (\prod_{i=1}^kg_ic^{2p}).\bp$ is
  an injection.
\end{lemma}
\begin{proof}
Consider a point $(\prod_{i=1}^kg_ic^{2p}).\bp$ in the image.
  Set $g_0:=c^{-2p}$. Suppose that for each $i$ we have
  $g_i=e_i^{-1}c^pe_i$ for $e_i\in G_1$.
  For $0\leqslant i\leqslant k$ set
  $z'_{2i}:=(\prod_{j=0}^ig_jc^{2p}).\bp$,
  $z_{2i}:=(\prod_{j=0}^ig_jc^{2p})c^{-2p}\!.\bp$, and
  $\Z_{2i}:=(\prod_{j=0}^ig_jc^{2p})\E$.
  For $0<i\leqslant k$ set
  $z_{2i-1}:=(\prod_{j=0}^{i-1}g_jc^{2p})e_{2i-1}^{-1}.\bp$,
  $z'_{2i-1}:=(\prod_{j=0}^{i-1}g_jc^{2p})e_{2i-1}^{-1}c^p\!.\bp$, and
  $\Z_{2i-1}:=(\prod_{j=0}^{i-1}g_jc^{2p})e_{2i-1}^{-1}\E$.
  See \fullref{fig:zigzag}.

  \begin{figure}[h]
    \labellist
    \tiny
    \pinlabel $z_0=c^{-2p}\!.\bp$ [br] at 35 3 
    \pinlabel $z_0'=\bp$ [bl] at 60 3
    \pinlabel $z_1=e_1^{-1}\!.\bp$ [tr] at 110 45
    \pinlabel $z_1'=e_1^{-1}c^p\!.\bp$ [tl] at 133 45
     \pinlabel $z_2$ [br] at 185 3 
     \pinlabel $z_2'$ [bl] at 225 3
      \pinlabel $z_{2k-1}$ [tr] at 280 43
    \pinlabel $z_{2k-1}'$ [tl] at 291 45
      \pinlabel $z_{2k}$ [br] at 300 3
      \pinlabel $(\prod_{i=1}^kg_ic^{2p}).\bp=z_{2k}'$ [bl] at 300 3
      \small
    \pinlabel $\Z_0=\E$ [t] at 35 0
    \pinlabel $\Z_1=e_1^{-1}\E$ [b] at 122 45
    \pinlabel $\Z_2=e_1^{-1}c^pe_1\E=g_1\E$ [t] at 206 0
    \pinlabel $\Z_{2k-1}$ [b] at 290 45
    \pinlabel $\Z_{2k}=(\prod_{i=1}^kg_ic^{2p})\E$ [t] at 320 0
    \endlabellist
    \centering
    \includegraphics[width=.9\textwidth]{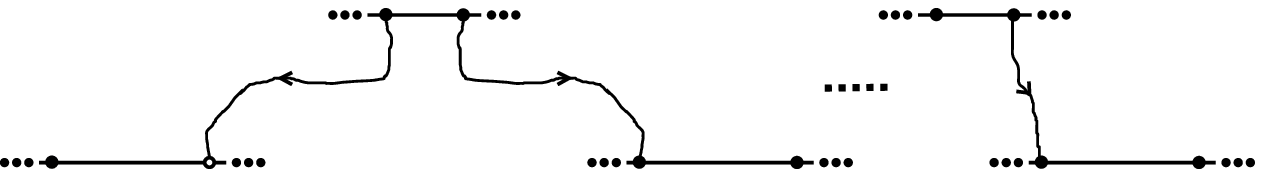}
    \caption{$(\prod_{i=1}^kg_ic^{2p}).\bp$}
    \label{fig:zigzag}
  \end{figure}
 
  Let us complete the proof assuming the following claim, to which we
  shall return:
\begin{equation}
  \label{eq:13}
  \forall 0\leqslant i<j\leqslant 2k,\quad
  d^\pi_{\Z_i}(z'_i,\Z_j)< 5K\quad\text{ and } \quad d^\pi_{\Z_j}(z_j,\Z_i)< 5K
\end{equation}

When $p$ is sufficiently large, $d(z_i,z_i')\gg 10K$ for all $i$, so \eqref{eq:13} implies that  $\Z_i\order\Z_j$ for
  all $0\leqslant i<j\leqslant 2k$, where $\order$ is the order of
  \fullref{order} on $\YY[\Z_0,\Z_{2k}]$.
  
Suppose that the map $\hat G_{4,p,D}\to X$ is not an injection; there exist distinct
  elements  $(g_1,\dots,g_m)$ and $(h_1,\dots,h_n)$ of $\hat G_{4,p,D}$ with the same image $z\in X$. 
  Suppose $m+n$ is minimal among such tuples.
If $h_1\E=g_1\E$ then $h_1=g_1$ by \fullref{sameaxissameelement}.
This contradicts minimality of $m+n$, so we must have $h_1\E\neq g_1\E$.
Let $\Z_0,\dots,\Z_{2m}$ be as in \fullref{fig:zigzag} for
$(g_1,\dots,g_m)$.
By definition, $\bp\in\Z_0$ and $z\in \Z_{2m}$.
By \eqref{eq:13}, $\pi_{\Z_{2m}}(\bp)$ is close to $z_{2m}$.
By \fullref{BGI}, any geodesic from $\bp$ to $z$ ends with a segment
that stays close to the subsegment of $\Z_{2m}$ between $z_{2m}$ and
$z=z'_{2m}$.
However, if $\Z'_0,\dots,\Z'_{2n}$ are as in \fullref{fig:zigzag} for
$(h_1,\dots,h_n)$, then the same is true for $\Z'_{2n}$, which implies
$d^\pi_{\Z_{2m}}(\Z'_{2n},\Z'_{2n})\gadd d(z_{2m},z'_{2m})=|c^{2p}|$.
Once $p$ is sufficiently large, \pzero requires
$\Z_{2m}=\Z'_{2n}$.
Thus, $\YY{[\Z_0,\Z_{2m}]}=\YY{[\Z'_0,\Z'_{2n}]}$, and all of the
$\Z_i$ and $\Z'_j$ are comparable in the order $\order$ on $\YY{[\Z_0,\Z_{2m}]}$.
In particular, $\Z'_2=h_1\E\neq g_1\E=\Z_2$, so one of them comes
before the other.
Suppose, without loss of generality, that $h_1\E\order g_1\E$.
Then $d_{h_1\E}(g_1\E,\Z_{2m})\leqslant \theta'$, by \fullref{order},
and $d^\pi_{h_1\E}(\Z_{2m},h_1c^{2p}\!.\bp)<5K$ by \eqref{eq:13},
so:
\begin{align*}
d^\pi_{h_1\E}(g_1.\bp,h_1c^{2p}\!.\bp)&\leqslant
                                      d^\pi_{h_1\E}(g_1\E,\Z_{2m})+d^\pi_{h_1\E}(\Z_{2m},h_1c^{2p}\!.\bp)\\
  &< \theta'+2\theta+5K<7K
\end{align*}
On the other hand, $d^\pi_{h_1\E}(\bp,h_1.\bp)<3K$, by \eqref{eq:6}, so
$d^\pi_{h_1\E}(\bp,g_1.\bp)\geqslant |c^{2p}|-10K\gg C'$.
By \fullref{BGI}, any geodesic from $\bp$ to $g_1.\bp$ passes within
distance $2C'$ of $\pi_{h_1\E}(g_1.\bp)$, which is less than $7K$ from $h_1c^{2p}\!.\bp$.
This means $g_1\in G_{4,p,(7K+2C')}'$, which is a
contradiction if $D\geqslant 7K+2C'$.
Thus, if $D\geqslant 7K+2C'$ then for sufficiently large $p$ the map is injective.

  \smallskip

We prove \eqref{eq:13} by induction on $m=j-i$.
For each $0\leqslant i<2k$ we have that $z_i'$ and $z_{i+1}$ differ by
an element of $G_1$, so $\Z_i\neq\Z_{i+1}$ and
$d^\pi_{\Z_{i+1}}(z_{i+1},z'_i)\leqslant 2K$.
Furthermore, by \pzero,
$d^\pi_{\Z_{i+1}}(\Z_i,\Z_i)\leqslant \theta$.
Thus:
  \[d^\pi_{\Z_{i+1}}(z_{i+1},\Z_i)\leqslant d^\pi_{\Z_{i+1}}(z_{i+1},z'_i)+d^\pi_{\Z_{i+1}}(z'_i,\Z_i)\leqslant d^\pi_{\Z_{i+1}}(z_{i+1},z'_i)+d^\pi_{\Z_{i+1}}(\Z_i,\Z_i)\leqslant 2K+\theta<3K\]
Similarly, $d^\pi_{\Z_i}(z'_i,\Z_{i+1})< 3K$.

Now extend $m$ to $m+1$: Suppose that for some $m\geqslant 1$ and all $0<j-i\leqslant m$ we have
$d^\pi_{\Z_j}(z_j,\Z_{i})< 5K$ and
$d^\pi_{\Z_i}(z'_i,\Z_{j})< 5K$. (Note that this implies $\Z_i\neq\Z_j$.)
Then for all $0\leqslant i\leqslant 2k-m-1$:
\[d_{\Z_{i+1}}(\Z_{i+m+1},\Z_i)\geqslant
  d^\pi_{\Z_{i+1}}(\Z_{i+m+1},\Z_i)-2\theta>
  d(z_{i+1},z'_{i+1})-10K-2\theta\gg\theta'\]
The final inequality is true for sufficiently large $p$, because 
the distance between $z_{i+1}$ and $z_{i+1}'$ is either $|c^p|$ or
$|c^{2p}|\eadd 2|c^p|$, according to whether $i$ is even or odd.
Thus, by \spthree and \spfour:
\[d^\pi_{\Z_i}(\Z_{i+m+1},\Z_{i+1})\leqslant
  d_{\Z_i}(\Z_{i+m+1},\Z_{i+1})+2\theta=d_{\Z_i}(\Z_{i+1},\Z_{i+1})+2\theta\leqslant\theta'+2\theta<2K\]
which
implies:
\[d^\pi_{\Z_i}(z'_i,\Z_{i+m+1})\leqslant
  d^\pi_{\Z_i}(z'_i,\Z_{i+1})+d^\pi_{\Z_i}(\Z_{i+1},\Z_{i+m+1})< 3K+2K=5K\]
A similar argument gives
$d^\pi_{\Z_{i+m+1}}(z_{i+m+1},\Z_{i})< 5K$.
This completes the induction.
\end{proof}

\begin{proof}[Proof of \fullref{main}]
  Take $D$ and $p$ as in \fullref{lemma:injection}.
  For this $D$, enlarge $p$ if necessary to satisfy the hypotheses of
  \fullref{lemma:divrate}.
  Set $G_4:=G_{4,p,D}$.
\end{proof}

\section{Questions}
\begin{question}\label{q1}
  Can we replace purely exponential growth of $G$ by divergence of $G$
  in \fullref{maintheorem}?
\end{question}
By \cite{MatYabJae15}, the answer is `yes' when $X$ is hyperbolic.

\medskip

Recall in \eqref{eq:11} we showed
$\Theta_{G}(s)$ is comparable to
$\Theta_{G_{3,p}}(s/2)$, while it is clear that $\Theta_{G_{3,p}}(s/2)\leqslant \Theta_N(s/2)$.
If $G$ is divergent then $\Theta_{G}(s)$ diverges at $s=\delta_G$, which means $\Theta_N(t)$
diverges at $t=\delta_G/2$.
There are two possible circumstances in which $\Theta_N(t)$ diverges at $t=\delta_G/2$:
\begin{equation}
  \label{eq:5}
  \text{Either }\delta_N>\delta_G/2, \text{ or } \delta_N=\delta_G/2
  \text{ and } N \text{ is divergent.}
\end{equation}
We proved the first case of \eqref{eq:5} directly, with the additional
assumption of purely exponential growth of $G$.
  The approach of \cite{MatYabJae15} is
  to prove, if $X$ is hyperbolic, that 
  $\delta_N=\delta_G$ when $N$ is divergent, so, since $\delta_G>\delta_G/2$, the second case of \eqref{eq:5} is
  impossible.
  Thus, a positive answer to \fullref{q1} would be implied by a
  positive answer to the following question, which is also interesting
  in its own right.
  \begin{question}\label{q2}
    If $G$ is a group acting properly by isometries
    with a strongly contracting element
    on a geodesic metric space $X$ and $G\act X$ is divergent, is it true that for every
    divergent normal subgroup $N$ of $G$ we have $\delta_N=\delta_G$?
  \end{question}

  Jaerisch and Matsuzaki \cite{JaeMat17} show that if $F$ is a finite
  rank free group and $N$ is a non-trivial normal subgroup of $F$
  then, with respect to a word metric defined by a free generating set
  of $F$, there is a inequality
  $\delta_N+\frac{1}{2}\delta_{F/N}\geqslant \delta_F$.
  Notice, $\delta_N>\delta_F/2$ by the lower cogrowth bound, and
  $\delta_{F/N}<\delta_F$ by growth tightness of $F$.

  \begin{question}
    Is there an analogue of Jaerisch and Matsuzaki's inequality for
    $G$ acting with a strongly contracting element and complementary
    growth gap? Note that we know both growth tightness, by
    \cite{ArzCasTao15}, and lower cogrowth bound, by \fullref{maintheorem},
    for such actions.
  \end{question}

  For $G=X=F_n$ \cite{MR599539,MR2141699,MR1233406} and
  $X=\mathbb{H}^2$ and $G$ a closed surface group \cite{MR2891734},
  there exists a sequence $(N_i)_{i\in\mathbb{N}}$ of normal subgroups of $G$ such that
  $\delta_{N_i}/\delta_G$ limits to $1/2$, so the lower cogrowth bound
  is optimal.

  \begin{question}
    Is the lower cogrowth bound optimal in \fullref{maintheorem}?
  \end{question}

  \medskip

  We must mention that the \emph{upper} cogrowth bound is also
  very interesting.
  Grigorchuk \cite{MR599539} and Cohen \cite{Coh82} showed that when
  $F$ is a finite rank free group, with respect to a word metric defined
  by a free generating set the upper cogrowth bound
  $\delta_N/\delta_F=1$ is achieved for $N\lhd F$ if and only if $F/N$
  is amenable.
  There have been several generalizations  \cite{MR783536,MR2166367,MR3220551,DouSha16,CouBoSam17} to growth rates defined with
  respect to an action $G\act X$, but the most general to date
  \cite{CouBoSam17} still requires $G$ to be hyperbolic, the action to be cocompact, and $X$ to be
  either a Cayley graph of $G$ or a CAT(-1) space. In the vein of our theorem,
  it would be very interesting to generalize such a result to a
  non-hyperbolic setting.

\enlargethispage*{\baselineskip} 

\providecommand{\href}[2]{#2}


\begin{thebibliography}{10}

\bibitem{ArzCasGrua}
G.~N. Arzhantseva, C.~H. Cashen, D.~Gruber, and D.~Hume,
  \emph{\href{https://arxiv.org/abs/1602.03767}{Negative curvature in graphical
  small cancellation groups}}, Groups Geom. Dyn. (in press).

\bibitem{ArzCasTao15}
G.~N. Arzhantseva, C.~H. Cashen, and J.~Tao,
  \emph{\href{http://dx.doi.org/10.2140/pjm.2015.278.1}{Growth tight actions}},
  Pacific J. Math. \textbf{278} (2015), no.~1, 1--49.

\bibitem{BesBroFuj17}
M.~Bestvina, K.~Bromberg, K.~Fujiwara, and A.~Sisto,
  \emph{\href{https://arxiv.org/abs/1711.08722}{Acylindrical actions on
  projection complexes}}, preprint (2017),
  \href{http://arXiv.org/abs/1711.08722v1}{{\texttt{arXiv:1711.08722v1}}}.

\bibitem{BesBroFuj15}
M.~Bestvina, K.~Bromberg, and K.~Fujiwara,
  \emph{\href{http://dx.doi.org/10.1007/s10240-014-0067-4}{Constructing group
  actions on quasi-trees and applications to mapping class groups}}, Publ.
  Math. Inst. Hautes \'Etudes Sci. \textbf{122} (2015), no.~1, 1--64.

\bibitem{MR2891734}
P.~Bonfert-Taylor, K.~Matsuzaki, and E.~C. Taylor,
  \emph{\href{https://doi.org/10.1007/s12220-010-9204-6}{Large and small covers
  of a hyperbolic manifold}}, J. Geom. Anal. \textbf{22} (2012), no.~2,
  455--470.

\bibitem{MR783536}
R.~Brooks, \emph{\href{https://doi.org/10.1515/crll.1985.357.101}{The bottom of
  the spectrum of a {R}iemannian covering}}, J. Reine Angew. Math. \textbf{357}
  (1985), 101--114.

\bibitem{MR1233406}
C.~Champetier, \emph{\href{https://doi.org/10.1112/blms/25.5.438}{Cocroissance
  des groupes \`a petite simplification}}, Bull. London Math. Soc. \textbf{25}
  (1993), no.~5, 438--444.

\bibitem{Coh82}
J.~M. Cohen, \emph{\href{https://doi.org/10.1016/0022-1236(82)90090-8}{Cogrowth
  and amenability of discrete groups}}, J. Funct. Anal. \textbf{48} (1982),
  no.~3, 301--309.

\bibitem{MR1214072}
M.~Coornaert,
  \emph{\href{http://projecteuclid.org/euclid.pjm/1102634263}{Mesures de
  {P}atterson-{S}ullivan sur le bord d'un espace hyperbolique au sens de
  {G}romov}}, Pacific J. Math. \textbf{159} (1993), no.~2, 241--270.

\bibitem{CouBoSam17}
R.~Coulon, F.~Dal'Bo, and A.~Sambusetti,
  \emph{\href{https://arxiv.org/abs/1709.07287}{Growth gap in hyperbolic groups
  and amenability}}, preprint (2017),
  \href{http://arXiv.org/abs/1709.07287v1}{{\texttt{arXiv:1709.07287v1}}}.

\bibitem{DalPeiPic11}
F.~Dal'Bo, M.~Peign{{\'e}}, J.-C. Picaud, and A.~Sambusetti,
  \emph{\href{http://dx.doi.org/10.1017/S0143385710000131}{On the growth of
  quotients of {K}leinian groups}}, Ergodic Theory Dynam. Systems \textbf{31}
  (2011), no.~3, 835--851.

\bibitem{MR1786869}
P.~de~la Harpe, \emph{Topics in geometric group theory}, Chicago Lectures in
  Mathematics, University of Chicago Press, Chicago, IL, 2000.

\bibitem{DouSha16}
R.~Dougall and R.~Sharp,
  \emph{\href{https://doi.org/10.1007/s00208-015-1338-1}{Amenability, critical
  exponents of subgroups and growth of closed geodesics}}, Math. Ann.
  \textbf{365} (2016), no.~3-4, 1359--1377.

\bibitem{MR599539}
R.~I. Grigorchuk, \emph{Symmetrical random walks on discrete groups},
  Multicomponent random systems, Adv. Probab. Related Topics, vol.~6, Dekker,
  New York, 1980, pp.~285--325.

\bibitem{GriDeL97}
R.~Grigorchuk and P.~de~la Harpe,
  \emph{\href{http://dx.doi.org/10.1007/BF02471762}{On problems related to
  growth, entropy, and spectrum in group theory}}, J. Dynam. Control Systems
  \textbf{3} (1997), no.~1, 51--89.

\bibitem{Jae15}
J.~Jaerisch, \emph{\href{https://doi.org/10.1007/s12220-013-9427-4}{A lower
  bound for the exponent of convergence of normal subgroups of {K}leinian
  groups}}, J. Geom. Anal. \textbf{25} (2015), no.~1, 298--305.

\bibitem{JaeMat17}
J.~Jaerisch and K.~Matsuzaki,
  \emph{\href{https://doi.org/10.1090/proc/13568}{Growth and cogrowth of normal
  subgroups of a free group}}, Proc. Amer. Math. Soc. \textbf{145} (2017),
  no.~10, 4141--4149.

\bibitem{Mat17}
K.~Matsuzaki, \emph{Growth and cogrowth tightness of {K}leinian and hyperbolic
  groups}, Geometry and Analysis of Discrete Groups and Hyperbolic Spaces
  (M.~Fujii, N.~Kawazumi, and K.~Ohshika, eds.), RIMS K\^{o}ky\^{u}roku
  Bessatsu, vol. B66, 2017, pp.~21--36.

\bibitem{MatYabJae15}
K.~Matsuzaki, Y.~Yabuki, and J.~Jaerisch,
  \emph{\href{https://arxiv.org/abs/1511.02664}{Normalizer, divergence type and
  {P}atterson measure for discrete groups of the {G}romov hyperbolic space}},
  preprint (2015),
  \href{http://arXiv.org/abs/1511.02664v1}{{\texttt{arXiv:1511.02664v1}}}.

\bibitem{MR2141699}
Y.~Ollivier,
  \emph{\href{http://aif.cedram.org/item?id=AIF_2005__55_1_289_0}{Cogrowth and
  spectral gap of generic groups}}, Ann. Inst. Fourier (Grenoble) \textbf{55}
  (2005), no.~1, 289--317.

\bibitem{MR2166367}
T.~Roblin, \emph{\href{https://doi.org/10.1007/BF02785371}{Un th\'eor\`eme de
  {F}atou pour les densit\'es conformes avec applications aux rev\^etements
  galoisiens en courbure n\'egative}}, Israel J. Math. \textbf{147} (2005),
  333--357.

\bibitem{MR3220551}
T.~Roblin and S.~Tapie,
  \emph{\href{http://dx.doi.org/10.1002/9781118557754.ch3}{Exposants critiques
  et moyennabilit\'e}}, G\'eom\'etrie ergodique, Monogr. Enseign. Math.,
  vol.~43, Enseignement Math., Geneva, 2013, pp.~61--92.

\bibitem{MR1792293}
Y.~Shalom, \emph{\href{https://doi.org/10.2307/2661380}{Rigidity, unitary
  representations of semisimple groups, and fundamental groups of manifolds
  with rank one transformation group}}, Ann. of Math. (2) \textbf{152} (2000),
  no.~1, 113--182.

\bibitem{MR882827}
D.~Sullivan,
  \emph{\href{http://projecteuclid.org/euclid.jdg/1214440979}{Related aspects
  of positivity in {R}iemannian geometry}}, J. Differential Geom. \textbf{25}
  (1987), no.~3, 327--351.

\bibitem{Yan16}
W.~Yang, \emph{\href{https://arxiv.org/abs/1612.03648}{Statistically
  convex-cocompact actions of groups with contracting elements}}, preprint
  (2016),
  \href{http://arXiv.org/abs/1612.03648v4}{{\texttt{arXiv:1612.03648v4}}}.

\end{thebibliography}
\end{document}